\documentclass[10pt, a4paper]{amsart}

\usepackage{amssymb,amsmath,amsthm}
\usepackage{tikz-cd}
\usetikzlibrary{arrows}
\usepackage{fullpage}
\usepackage{graphicx}
\usepackage{hyperref}
\usepackage{mathtools}
\usepackage[titletoc,title]{appendix}
\usepackage{graphicx}
\usepackage{caption}
\usepackage{subcaption}
\usepackage{blindtext}
\usepackage[titletoc,title]{appendix}
\usepackage[english]{babel}
\usepackage{longtable}
\usepackage{mathrsfs}
\usepackage[ruled,vlined]{algorithm2e}
\usepackage{cite}

\newtheorem{thm}{Theorem}[section]
\newtheorem{lma}[thm]{Lemma}
\newtheorem{cor}[thm]{Corollary}
\newtheorem{prop}[thm]{Proposition}
\theoremstyle{definition}
\newtheorem{defn}[thm]{Definition}

\newtheorem*{thm*}{Theorem}

\newcommand{\fakeenv}{} 

\newenvironment{restate}[2]  
{ 
 \renewcommand{\fakeenv}{#2} 
 \theoremstyle{plain} 
 \newtheorem*{\fakeenv}{#1~\ref{#2}} 
 \begin{\fakeenv}
}
{
 \end{\fakeenv}
}

\makeatletter
\def\Ddots{\mathinner{\mkern1mu\raise\p@
\vbox{\kern7\p@\hbox{.}}\mkern2mu
\raise4\p@\hbox{.}\mkern2mu\raise7\p@\hbox{.}\mkern1mu}}
\makeatother

\newcommand{\Z}{\mathbb{Z}}

\DeclareSymbolFontAlphabet{\amsmathbb}{AMSb}

\DeclareMathOperator{\sub}{Sub}

\DeclareMathOperator{\ap}{Arith}

\DeclareMathOperator{\word}{\mathsf{word}}
\DeclareMathOperator{\red}{\mathsf{red}}

\DeclareMathOperator{\core}{Core}
\DeclareMathOperator{\rcore}{Core^{rel}}

\DeclarePairedDelimiter\abs{\lvert}{\rvert}

\makeatletter
\let\oldabs\abs
\def\abs{\@ifstar{\oldabs}{\oldabs*}}

\title{The fully compressed subgroup membership problem}
\date{\today}
\author{Marco Linton}

\address{Mathematics institute, Zeeman building, university of Warwick, Coventry, CV4 7AL}

\email{marco.linton@warwick.ac.uk}

\begin{document}

\maketitle

\begin{abstract}
Suppose that $F$ is a free group and $k$ is a natural number. We show that the fully compressed membership problem for $k$-generated subgroups of $F$ is solvable in polynomial time. In order to do this, we adapt the theory of Stallings' foldings to handle edges with compressed labels. This partially answers a question of Markus Lohrey.
\end{abstract}

\section{Introduction}

The rational subset membership problem for free monoids is a classic problem in formal language theory. This problem naturally extends to the world of group theory as follows. Let $\Sigma$ be a finite set, let $G$ be a group and $\pi:\Sigma^*\to G$ a surjective morphism. Then the \emph{rational subset membership problem for $G$} is to decide, given as input a finite state automaton $\mathcal{A}$ over $\Sigma$ and a word $w\in \Sigma^*$, whether $\pi(w) \in \pi(L(\mathcal{A}))$. This problem is known to be solvable when $G$ is free \cite{benois_69}, abelian \cite{grunschlag_99} or a Droms right angled Artin group \cite{lohrey_08}. The first two classes also admit polynomial time solutions. An important class of rational subsets of groups is that of subgroups. The \emph{subgroup membership problem} is the rational subset membership problem when restricted to this class. The full rational subset membership problem is strictly harder: the subgroup membership problem is solvable for nilpotent groups \cite{avenhaus_89}, but there are nilpotent groups in which the rational subset membership problem is undecideable \cite{romankov_99}.



Often, requiring inputs to be words over the generators may not be the most natural thing to do. Take, for instance, the linear groups. Each element of a linear group can be represented by a matrix with binary integer entries, an exponentially more succinct representation. In \cite{gurevich_05}, a variant of the subgroup membership problem was considered where the inputs could contain powers $x^i$ where $x\in \Sigma$ and $i\in \Z$ is encoded in binary. The authors show that this variant remains solvable in polynomial time in the class of free groups. As a consequence, they also show that the subgroup membership problem for $\text{PSL}(2, \Z)$ is solvable in polynomial time, even when the input matrix entries are encoded in binary. In \cite{lohrey_21}, the subgroup membership problem for free groups was further generalised so that the input could contain power words $w^i$ where $w\in \Sigma^*$ and $i\in \Z$ is encoded in binary. The polynomial time solution to this variant was similarly used to provide a polynomial time solution to the membership problem in $\text{GL}(2, \Z)$, even when the input matrix entries are encoded in binary. Another version of the compressed membership problem can be found in \cite{myasnikov_11}, arising as an intermediate step to solve the word problem in polynomial time in the Baumslag--Gersten group.

A more general way of succinctly representing elements of $\Sigma^*$ is given by \emph{straightline programs} (SLPs): context free grammars that generate exactly one word. In \cite{plandowski_99} it was asked whether the rational subset membership problem for free monoids was solvable in polynomial time when the inputs were compressed using SLPs. In \cite{jez_12}, Artur Je\.{z} settled this question in the affirmative. In \cite{lohrey_21} and \cite{lohrey_21_compression}, Markus Lohrey asks whether this problem is solvable in polynomial time for subgroups of free groups. We call this the \emph{fully compressed subgroup membership problem} for free groups. In the case that the number of input generators is fixed, we answer Markus Lohrey's question in the affirmative.

\begin{restate}{Corollary}{main_cor}
The fully compressed membership problem for $k$-generated subgroups of a free group is in P.
\end{restate}

This result follows directly from our stronger result, Theorem \ref{main_subgroup}. See Algorithm \ref{CDFA_algorithm} for the sketch of our compressed Stallings' folding algorithm.

\subsection*{Acknowledgments} We would like to thank Saul Schleimer for suggesting the problem and for the many helpful mathematical discussions.

\section{Preliminaries}

We will fix $\Sigma$ to be a finite alphabet and $\Sigma^*$ to be the free monoid generated by $\Sigma$. A \emph{letter} is an element of $\Sigma$; a \emph{word} is an element of $\Sigma^*$. The symbol $\epsilon$ will denote the \emph{empty word}.

A \emph{factorisation} of a word $w\in \Sigma^*$ is an equality $w = w_0\cdot w_1\cdot \ldots\cdot w_n$ where $w_i\in \Sigma^*$ for all $0\leq i\leq n$. Now let $w = w_0\cdot w_1\cdot\ldots\cdot w_n\in \Sigma^*$ be the unique factorisation with $w_i\in \Sigma$. Given $0\leq i\leq j\leq n$, we make the following definitions:
\begin{enumerate}
\item $w[i] = w_i$ is the $i+1^{\text{th}}$ letter of $w$,
\item $\abs{w} = n+1$ is the \emph{length} of $w$,
\item $w[i:j] = w_i\cdot w_{i+1}\cdot \ldots\cdot w_{j-1}$ is a \emph{subword},
\item $w[:i] = w_0\cdot w_1\cdot \ldots\cdot w_{i-1}$ is a \emph{prefix},
\item $w[j:] = w_j\cdot w_{j+1}\cdot \ldots\cdot w_n$ is a \emph{suffix},
\end{enumerate}

Let $x, w\in \Sigma^*$ be words, a \emph{left $w$--factorisation of length $k$} of $x$ is a factorisation of the form:
\[
x = w[i:]\cdot w^n\cdot w[:j]\cdot z
\]
such that $\abs{w[i:]\cdot w^n\cdot w[:j]} = k$. We define \emph{right $w$-factorisations} analogously. 

We write $\ap(j, k, l)$ for the arithmetic progression $\{i\cdot j + l \mid 0\leq i\leq k\}$.

Let $\Sigma^{-1}$ denote the set of formal inverses of $\Sigma$. The free group $F(\Sigma)$, freely generated by $\Sigma$, comes equipped with a natural surjective monoid homomorphism $\pi:(\Sigma\sqcup\Sigma^{-1})^*\to F(\Sigma)$. The map $\pi$ has a section $\red:F(\Sigma)\to (\Sigma\sqcup\Sigma^{-1})^*$ mapping each element $g\in F(\Sigma)$ to the unique freely reduced word in $\pi^{-1}(g)$.

\subsection{Compression}

We refer the reader to \cite{lohrey_12} for an excellent survey on straight-line programs and compressed finite state automata. A \emph{straight-line program}, or SLP, is a tuple $\mathbb{X} = \langle \Sigma, \mathcal{X}, X_n, \mathcal{P}\rangle$ consisting of the following: 
\begin{enumerate}
\item $\Sigma$ is a finite alphabet of \emph{terminal} letters,
\item $\mathcal{X} = \{X_1, \ldots, X_n\}$ is a finite alphabet of \emph{non-terminal} letters,
\item $X_n$ is the \emph{root} non-terminal,
\item $\mathcal{P} = \{X_i\to W_i\}_{i = 1}^n$ is the set of \emph{production rules} where $W_i\in (\Sigma\cup \{X_1, \ldots, X_{i-1}\})^*$.
\end{enumerate}

We will denote by $\word{\mathbb{X}}\in \Sigma^*$ the word obtained by repeatedly replacing each non-terminal with its production. The \emph{height} of a non-terminal $X_i\in \mathcal{X}$, denoted by $||X_i||$, is inductively defined as follows: the height of a terminal is zero and the height of a non-terminal is the maximum height of the symbols appearing in its production, plus one.

If $X$ is a non-terminal and $i$ and $j$ are positions in $\word{X}$, then we may write $X[i:j]$; we call this a \emph{truncated non-terminal}. The intention is that we have $\word{(X[i:j])} = \word{(X)}[i:j]$. A \emph{composition system} $\mathbb{X} = \langle \Sigma, \mathcal{X}, X_n, \mathcal{P}\rangle$ is defined in the same way as an SLP except that production rules can contain truncated non-terminals. The expressive power of SLPs and composition systems is virtually the same; this is because a composition system can be transformed into an equivalent SLP in quadratic time \cite{hagenah_00}.

We abuse notation and factorise SLPs just as we factorise words. That is, if $W$ is a non-terminal, then we write $W = W_0\cdot W_1\cdot\ldots\cdot W_n$ when we mean $\word{W} = \word{W_0}\cdot \word{W_1}\cdot\ldots\cdot \word{W_n}$. Thus, if $X$ and $W$ are non-terminals, then a \emph{left $W$-factorisation of $X$} is a factorisation of the form:
\[
X = W[i:]\cdot W^n\cdot W[:j]\cdot Z
\]
where $Z$ is also a non-terminal.

\begin{prop}
\label{SLP_left_right}
There is a polynomial-time algorithm that, given as input two non-terminals $X$ and $W$, decides if $X$ has a left (or right) $W$-factorisation of length $k\geq |\word{W}|$. If it does, then the algorithm also computes a maximal left (or right) $W$-factorisation of $X$.
\end{prop}

\begin{proof}
We prove the result for left $W$-factorisations. We first assume that $|\word{X}| = |\word{W}|$. Hence the problem becomes to decide if there is some integer $i$ such that 
\[
X = W[i:]\cdot W[:i].
\]
This is the conjugacy problem and can be solved in polynomial time by Theorem 3.7 in \cite{saul_08}. This algorithm also produces a conjugating word. Equivalently, the index $i$. So now suppose $|\word{X}| > |\word{W}|$. By the above, we can decide if $X[:\abs{\word{W}}]$ has a $W$-factorisation of length $|\word{W}|$. Moreover, if so, we may also compute an index $i$ such that $X[:i] = W[-i:]$. Then by computing the largest prefix that $X[i:]$ and $W^{\left\lceil\frac{|\word{W}|}{|\word{X}|}\right\rceil}$ have in common gives us the required factorisation. This may be done in polynomial time \cite{gasieniec_96}.
\end{proof}

Let $X$ and $W$ be non-terminals and $j$ an integer. We will say $X$ crosses $W$ at $i$ if there is an integer $j$ such that $j\leq i<j+\abs{\word{X}}$ and 
\[
W[j:j+\abs{\word{X}}] = X.
\]
Then we denote by $\sub(X, W, i)$ the set of such integers $j$. We say $X$ \emph{crosses} the production $W\to U\cdot V$, if 
\[
X = U[-i_1:]\cdot V[:i_2]
\]
for some $i_1\geq 0$, $i_2> 0$. The following is Lemma 1 in \cite{lifshits_07}.

\begin{lma}
\label{crossing_appearances}
There is a polynomial-time algorithm that, given as input non-terminals $X$ and $W$ and an integer $i\leq \abs{\word{W}}$, computes integers $j, k, l$ such that $\sub(X, W, i) = \ap(j, k, l)$.
\end{lma}

\subsection{Compressed automata}

A \emph{compressed non-deterministic finite state automaton}, or CNFA, is a tuple $\mathcal{A} = (Q, \Sigma, \mathcal{X}, \delta, \mathcal{P}, q_0, F)$, where:
\begin{enumerate}
\item $Q$ is a finite set of \emph{states},
\item $\Sigma$ is a finite alphabet of terminal letters,
\item $\mathcal{X} = \{X_1, \ldots, X_n\}$ is a finite alphabet of non-terminal letters,
\item $\delta\in Q\times\mathcal{X}\times Q$ is the set of \emph{transitions},
\item $\mathcal{P} = \{X_i\to W_i\}_{i=1}^n$ is the set of \emph{production rules} where $W_i\in (\Sigma\sqcup \{X_1, \ldots, X_{i-1}\})^*$,
\item $q_0\in Q$ is the \emph{initial state},
\item $F\subset Q$ is the set of \emph{final states}.
\end{enumerate}
A CNFA $\mathcal{A}$ is a \emph{compressed deterministic finite state automaton}, or CDFA, if for each pair of transitions $\alpha = (p, X, q)$, $\beta = (p, Y, r)\in \delta$, $X[0] = Y[0]$ implies that $\alpha = \beta$. Often we will allow our productions to also contain truncated non-terminals. This only affects the run-time of any algorithm presented by a quadratic factor.

A word $w \in \Sigma^*$ is \emph{accepted} by $\mathcal{A}$ if there is a sequence of states $q_0, q_1, \ldots, q_k$ such that $(q_{i-1}, X_i, q_i)\in \delta$ for all $1\leq i\leq k$ and $w = \word{X_1}\cdot \word{X_2}\cdot \ldots \word{X_k}$. The language $L(\mathcal{A})$ of $\mathcal{A}$ is the set of all words accepted by $\mathcal{A}$.

We will say that a non-terminal $X$ \emph{determines a path from $\alpha(i)$ to $\beta(j)$} in $\mathcal{A}$, where $\alpha$, $\beta\in \delta$, if there is a sequence of transitions $\alpha = (p_0, X_1, p_1), (p_1, X_2, p_2), \ldots, (p_{n-1}, X_n, p_n) = \beta$ such that
\[ 
\word{X} = \word{X_1}[i:]\cdot \word{X_2}\cdot \ldots \cdot \word{X_n[:j]}.
\]

The key result we will be using throughout this article is the main result from \cite{jez_12}:

\begin{thm}
\label{jez_compressed}
The fully compressed membership problem for CDFA (respectively CNFA) is P-complete (respectively NP-complete).
\end{thm}

An \emph{involutive CNFA} is a CNFA $\mathcal{A} = (Q, \Sigma\sqcup\Sigma^{-1}, \mathcal{X}, \delta, \mathcal{P}, q_0, F)$ equipped with an involution $^{-1}:\delta\to \delta$ such that if $(p, X, q)^{-1} = (q, Y, p)$ then $\word{X} = \word{Y}^{-1}$. We say $\mathcal{A}$ is an \emph{involutive CDFA} if it is a CDFA and if $\word{X}$ is a freely reduced word for each $X\in \mathcal{X}$. Denote by $\partial \mathcal{A}\subset Q$ the set of states with only one incoming and one outgoing transition.

A subset $H\subset F(\Sigma)$ of a free group is a \emph{rational subset} if there exists a CNFA $\mathcal{A}$ over $\Sigma\sqcup \Sigma^{-1}$ satisfying $H = \pi(L(\mathcal{A}))$. An important class of rational subsets of free groups are their finitely generated subgroups. The well developed theory of \emph{Stallings automata} allows us to solve the membership problem for subgroups efficiently \cite{sta_83,kapovich_02}. However, for our context, we will need to define a compressed analog. 

\begin{defn}
A \emph{compressed Stallings automaton} for a subgroup $H<F(\Sigma)$ is an involutive CDFA $\mathcal{A}$ such that $(\red\circ\pi)(L(\mathcal{A})) = \red(H)$ and such that the number of states and transitions are minimal.
\end{defn}

\section{Language intersections for CDFA}

Given two CDFAs $\mathcal{A}$ and $\mathcal{B}$, in this section we aim to understand what questions we can answer about $L(\mathcal{A})\cap L(\mathcal{B})$ in polynomial time. Given an SLP $\mathbb{X}$, we may decide if $\word{\mathbb{X}}\in L(\mathcal{A})\cap L(\mathcal{B})$ in polynomial time by Theorem \ref{jez_compressed}. However, deciding whether there exists a non-trivial element $w\in L(\mathcal{A})\cap L(\mathcal{B})$ in polynomial time requires more work. We call this problem the \emph{CDFA non-emptiness intersection problem} and show that it is in NP:

\begin{thm}
\label{compressed_nei}
The CDFA non-emptiness intersection problem for two CDFA is in NP.
\end{thm}

\begin{proof}
Let $\mathcal{A} = (Q_A, \Sigma, \mathcal{X}_A, \delta_A, \mathcal{P}_A, q_A, F_A)$ and $\mathcal{B} = (Q_B, \Sigma, \mathcal{X}_B, \delta_B, \mathcal{P}_B, q_B, F_B)$ be two CDFAs. After performing minor modifications to $\mathcal{A}$ and $\mathcal{B}$, we may assume that $q_A$ and $q_B$ don't have any incoming transitions and that each state in $F_A$ and $F_B$ doesn't have any outgoing transitions. We may further assume that there are no states with exactly one incoming transition and exactly one outgoing transition.

The proof will be based off the results of \cite{me_21} so we adopt the same notation. Let $\Delta$ be a graph with a single vertex and $E(\Delta) = \Sigma$. Then $\mathcal{A}$ and $\mathcal{B}$ correspond to marked directed graph maps $\gamma:\Gamma\to \Delta$ and $\lambda:\Lambda\to \Delta$ as follows: each vertex in $\Gamma$ and $\Lambda$ corresponds to a state in the decompressed automata for $\mathcal{A}$ and $\mathcal{B}$ respectively. Each edge in $\Gamma$ and $\Lambda$ corresponds to a transition in the decompressed automata for $\mathcal{A}$ and $\mathcal{B}$ respectively. Each edge in $\Gamma$ and $\Lambda$ maps to the edge in $\Delta$ corresponding with the label of their associated transitions. If $\Gamma$ is a graph, then $\bar{V}(\Gamma)\subset V(\Gamma)$ is the set of vertices with either indegree or outdegree different to 1. Then $\bar{E}(\Gamma)$ are the minimal segments connecting these vertices. Thanks to our initial assumptions, we have a correspondence between elements in $\bar{E}(\Gamma)$ and elements of $\delta_A$, 
\[
e\to \alpha(e) = (p(e), X(e), q(e)). 
\]
Similarly for $\bar{E}(\Lambda)$ and $\delta_B$, 
\[
f\to \beta(f) = (r(f), Y(f), s(f)).
\]

Now let $\Theta = \rcore(\Gamma\times_{\Delta}\Lambda)$ and let $\theta:\Theta\to \Delta$ be the natural map. If $L(\mathcal{A})\cap L(\mathcal{B})\neq \emptyset$, then there is some path $g:I\to \Theta$ starting at $(q_A, q_B)$ and ending at some vertex in $F_A\times F_B$. We may assume that $g$ has minimal length among all such paths. Denote by $g = g_1*\ldots*g_n$ the factorisation of $g$ such that $g_i\in \bar{E}(\Theta)$ for all $i$. Let $s_{\Gamma}:\mathbb{S}_{\Gamma}\to \Gamma$ and $s_{\Lambda}:\mathbb{S}_{\Lambda}\to \Lambda$ be the maps from Theorem 4.19 in \cite{me_21} and recall that $\Theta^s = \Theta - s(\mathbb{S})$ where $s:\mathbb{S} = \core(\mathbb{S}_{\Gamma}\times_{\Delta}\mathbb{S}_{\Lambda})\to \Theta$ is the natural map. Then by definition of $\Theta^s$, either $g_i\subset \Theta^s$ or $g_i$ lifts to some cycle in $s:\mathbb{S}\to \Theta$. Furthermore, by Lemma 3.7 in \cite{me_21}, $n\leq 4\cdot (\abs{\bar{E}(\Gamma)}\cdot \abs{\bar{E}(\Lambda)} + 1)$. Let $i\leq n$, we have two cases to consider.

Firstly, suppose $g_i\subset \Theta^s$. By Theorem 5.10 in \cite{me_21}, there is a sequence of non-terminals $X(e_1), \ldots, X(e_k)$, $Y(f_1), \ldots, Y(f_k)$ and integers $a_1, b_1, c_1, d_1, \ldots, a_k, b_k, c_k, d_k$ such that:
\[
\theta\circ g_i = \word(X(e_1)[a_1:b_1]\cdot Y(f_1)[c_1:d_1]\cdot\ldots\cdot X(e_k)[a_k:b_k]*Y(f_k)[c_k:d_k])
\]
and $k\leq 1008\cdot \left(\abs{\bar{E}(\Gamma)}\cdot \abs{\bar{E}(\Lambda)}\right)^2$.

Now suppose $g_i$ lifts to some $S^1\subset \mathbb{S}$. Then by Lemma 4.18 in \cite{me_21}, there is some $e\in \bar{E}(\Gamma)$ and integers $a, b, c, d, k$ such that:
\[
\theta\circ g_i = \word(X(e)[c:b]\cdot X(e)[a:b]^k\cdot X(e)[a:d]).
\]
Furthermore, $k\leq \abs{E(\Gamma)}\cdot \abs{E(\Lambda)}$.

It follows that there is an SLP $\mathbb{W}$ such that $\word{\mathbb{W}} = \theta\circ g$ and such that the size of $\mathbb{W}$ is polynomially bounded by the sizes of the input. Hence we can non-deterministically guess $\mathbb{W}$ and check if $\word{\mathbb{W}}\in L(\mathcal{A})$ and $\word{\mathbb{W}}\in L(\mathcal{B})$ in polynomial time by Theorem \ref{jez_compressed}.
\end{proof}

In the proof of Theorem \ref{compressed_nei}, we actually show that given any two CDFA $\mathcal{A}$ and $\mathcal{B}$, if $L(\mathcal{A})\cap L(\mathcal{B})\neq \emptyset$, there exists an SLP $\mathbb{W}$ such that $\word{\mathbb{W}}\in L(\mathcal{A})\cap L(\mathcal{B})$ and such that the size of $\mathbb{W}$ is polynomially bounded in terms of the sizes of $\mathcal{A}$ and $\mathcal{B}$. However, the proof does not suggest any method of finding $\mathbb{W}$ deterministically. One case in which we may find such an SLP is when our input alphabets consist of only a single symbol. Indeed, using integer arithmetic we may even compute a CNFA $\mathcal{C}$ such that $L(\mathcal{C}) = L(\mathcal{A})\cap L(\mathcal{B})$.

\begin{prop}
\label{unary_alphabet_intersection}
There is a polynomial-time algorithm that, give as input CDFAs $\mathcal{A}$ and $\mathcal{B}$ over a unary alphabet $\Sigma = \{a\}$, computes integers $n_0, \ldots, n_k, n$ such that:
\[
L(\mathcal{A})\cap L(\mathcal{B}) = a^{n_0}\cdot (a^n)^*\cdot (a^{n_1} \mid \ldots \mid a^{n_k}).
\]
\end{prop}

\begin{proof}
A CDFA over a unary alphabet must have at most one outgoing transition from each state and thus can only take one of the following forms: a single state, a segment, a cycle or a segment with a cycle attached at the end. By Lemma 2.7 in \cite{saul_08} we may compute for each SLP $\mathbb{X}$ over $\{a\}$ an integer $p$ such that $\word{\mathbb{X}} = a^p$. Thus, it is not hard to see that we may compute in polynomial time integers $p_0, \ldots, p_m, p$ and $q_0, \ldots, q_n, q$ such that:
\begin{align*}
L(\mathcal{A}) &= a^{p_0}\cdot (a^p)^*\cdot (a^{p_1} \mid \ldots \mid a^{p_m}),\\
L(\mathcal{B}) &= a^{q_0}\cdot (a^q)^*\cdot (a^{q_1} \mid \ldots \mid a^{q_n}),
\end{align*}
where $p_1, \ldots, p_m<p$ and $q_1, \ldots, q_n<q$. Now computing the integers $n_1, \ldots, n_k, n$ involves solving some systems of linear equations which may be done in polynomial time.
\end{proof}

We show one more case in which we may find $\mathbb{W}$ deterministically. Just like the unary language case, regular sublanguages of languages of the form $u\cdot(v)^*$, where $u, v\in \Sigma^*$, have very convenient representations using integers, along with our original words $u$ and $v$. We show that if $L(\mathcal{B})$ is of this form, then we may compute a CDFA $\mathcal{C}$ such that $L(\mathcal{C}) = L(\mathcal{A})\cap L(\mathcal{B})$. First, we shall need a Lemma.

\begin{lma}
\label{power_pump}
Let $\mathcal{A}$ be a CNFA and let $u, v\in\Sigma^*$ such that $\abs{v}>\abs{\word{X}}$ for all $X\in \mathcal{X}$. If $L(\mathcal{A})\cap u\cdot (v)^+\neq \emptyset$, then $u\cdot v^n\in L(\mathcal{A})$ for some $n\leq 4\cdot \abs{\delta} + 1$. If $\mathcal{A}$ is a CDFA, then $n\leq 2\cdot \abs{\delta} + 1$.
\end{lma}

\begin{proof}
Let $\Delta$ be a graph with a single vertex and $E(\Delta) = \Sigma$. Then just as in the proof of Theorem \ref{compressed_nei}, $\mathcal{A}$ determines a directed graph map $\gamma:\Gamma\to \Delta$. Recall from \cite{me_21} that $\gamma$ is a forwards immersion if for each vertex $V(\Gamma)$ and for each pair of outgoing edges $e, f\in E(\Gamma)$, we have $\gamma(e)\neq \gamma(f)$. Hence, if $\mathcal{A}$ is deterministic, then $\gamma$ is a forwards immersion. Now the result follows from Lemma 4.12 in \cite{me_21}.
\end{proof}

Note that the proof of Lemma \ref{power_pump} does not use the fact that $\mathcal{A}$ has compressed transition labels.

\begin{thm}
\label{intersection}
There exists a polynomial-time algorithm that, given as input a CDFA $\mathcal{A}$ and SLPs $\mathbb{U}$ and $\mathbb{V}$, computes a collection of integers $\{n_0, n_1, \ldots, n_k, n\}$ with $k \leq |F|$ such that
\[
L(\mathcal{A})\cap \word{\mathbb{U}}\cdot (\word{\mathbb{V}})^* = \word{\mathbb{U}}\cdot \word{\mathbb{V}}^{n_0}\cdot(\word{\mathbb{V}}^n)^*\cdot (\word{\mathbb{V}}^{n_1} \mid \ldots \mid \word{\mathbb{V}}^{n_k}).
\]
\end{thm}

\begin{proof}
By Theorem 2 in \cite{lifshits_07}, we may compute an SLP $\mathbb{Y}$ and an integer $k\geq 1$ such that $\word{\mathbb{Y}}$ is primitive and $\mathbb{V} = \mathbb{Y}^k$ in polynomial time. We may compute in polynomial time the largest integer $l\geq 0$ such that
\[
\word{\mathbb{U}} = \word{\mathbb{U}}[:-\abs{\word{\mathbb{Y}^l}}]\cdot \word{\mathbb{Y}}^l
\]
by Theorem 2.9 in \cite{saul_08}.

Now, if we can compute:
\[
L(\mathcal{A})\cap \word{\mathbb{U}}[:-\abs{\word{\mathbb{Y}^l}}]\cdot (\word{\mathbb{Y}})^*
\]
in polynomial time, then we can compute $L(\mathcal{A})\cap \word{\mathbb{U}}\cdot (\word{\mathbb{V}})^*$ in polynomial time by Proposition \ref{unary_alphabet_intersection}. Hence, from now on we may assume that $l = 0$ and $k = 1$.

We now add auxiliary symbols to our alphabet, $u$ and $v$, and will modify our automaton $\mathcal{A}$ in five steps: 

\begin{description}
\item[Step 1\label{itm:step1}] For each transition $(p, X, q)\in \delta$, we check using Proposition \ref{SLP_left_right} if $X$ has maximal right $\mathbb{V}$-factorisation of length greater than or equal to $\abs{\word{\mathbb{V}}}$. If so, then let $X = X[:l]\cdot \mathbb{V}[-i:]\cdot \mathbb{V}^k\cdot \mathbb{V}[:j]$ be the maximal right $\mathbb{V}$-factorisation found. Note that this is unique as $\word{\mathbb{V}}$ is primitive. We now add two new states $c$ and $d$ and four new transitions: $(p, X[:l+i], c)$, $(c, v^k, d)$, $(c, X[l+i:\abs{\word{X}} - j], d)$ and $(d, X[-j:], q)$. Now we remove the transition $(p, X, q)$ and denote by $\mathcal{A}'$ the resulting CDFA. Note that $\abs{\delta'}\leq 4\cdot \abs{\delta}$.
\item[Step 2\label{itm:step2}] For every pair of states $(p, q) \in Q'\times Q'$, decide if $\mathbb{V}^m$ determines a path from $p$ to $q$ for some $m\leq 2\cdot \abs{\delta'}+1$ using Theorem \ref{jez_compressed}. If so, then let $m$ be the smallest such integer and add a transition $(p, v^m, q)$, if it doesn't exist already.
\item[Step 3\label{itm:step3}] For every state $q\in Q'$, decide if $\mathbb{U}\cdot \mathbb{V}^m$ determines a path from $q_0'$ to $q$ for some $m\leq 2\cdot \abs{\delta'}+1$ using Theorem \ref{jez_compressed}. If so, then let $m$ be the smallest such integer and add a transition $(q_0', u\cdot v^m, q)$.
\end{description}

Denote by $\mathcal{A}''$ the resulting CDFA and by $\delta_v\subset \delta''$ the transitions with label in $\{u, v\}^*$. The above steps may be done in polynomial time and the subautomaton on the transitions $\delta'' - \delta_v$ accepts precisely the same language as the automaton $\mathcal{A}$ that we started off with. 

\textbf{Claim 1:} $\mathbb{V}^m$ determines a path between states $p$ and $q$ if and only if $v^m$ does.

One direction is by construction so we show the other direction by induction on $m$. The base case when $m = 0$ is trivial. So suppose the claim holds for all $k<m$. Suppose that $\mathbb{V}^m$ determines a path between states $p$ and $q$ traversing a transition $(c, X, d)\in \delta'' -\delta_z$ satisfying $\abs{\word{X}}\geq \abs{\word{\mathbb{V}}}$. Then $X$ has a $\mathbb{V}$-factorisation and so by \nameref{itm:step1}, $X = \mathbb{V}^k$ for some $k\geq 1$. By uniqueness of $\mathbb{V}$-factorisations, it follows that there is some $j<m$ such that $\mathbb{V}^j$ determines a path between $p$ and $c$. Hence $\mathbb{V}^{m-j-k}$ also determines a path between $d$ and $q$. By induction the claim follows. So now assume that $\mathbb{V}^m$ traverses only transitions $(c, X, d)$ satisfying $\abs{\word{X}}<\abs{\word{\mathbb{V}}}$. By Lemma \ref{power_pump}, we have $m\leq 2\cdot \abs{\delta'' - \delta_v}+1$. But now by \nameref{itm:step2}, there is a transition $(p, v^m, q)$ and so the claim is proven.

\textbf{Claim 2:} $\mathbb{U}\cdot \mathbb{V}^m$ determines a path between states $q_0$ and $q$ if and only if $u\cdot v^m$ does.

We also prove this claim by induction on $m$. Suppose that $\mathbb{U}\cdot \mathbb{V}^m$ determines a path between $q_0$ and $q$. If $\mathbb{U}$ determines a path between $q_0$ and some state, then we are done. So suppose $\mathbb{U}$ determines a path between $q_0$ and $\alpha(i)$ for some $\alpha\in \delta'' - \delta_v$ and some $i\geq 1$. As before, we first assume that $\mathbb{V}^m$ determines a path from $\alpha(i)$ to $q$ traversing a transition $(c, X, d)\in \delta'' - \delta_v$ satisfying $\abs{\word{X}}\geq \abs{\word{\mathbb{V}}}$. Note that $\mathbb{V}^m$ must traverse a transition by \nameref{itm:step1}. Then just as in the proof of claim 1, we may use induction to see that $u\cdot v^m$ must determine a path from $q_0$ to $q$. Now suppose that $\mathbb{V}^m$ traverses only transitions $(c, X, d)$ satisfying $\abs{\word{X}}<\abs{\word{\mathbb{V}}}$. Then by Lemma \ref{power_pump} and the fact that $\abs{\alpha[i:]}<\abs{\word{\mathbb{V}}}$, we have $m\leq 2\cdot \abs{\delta'' - \delta_v} + 1$. By \nameref{itm:step3}, there is a transition $(q_0, u\cdot v^m, q)$ already and the clam is proven.

Since $\mathcal{A}$ was assumed deterministic, for every pair of transitions $(p, v^i, q)$, $(p, v^j, r)$ with $i\leq j$, if we remove $(p, v^j, r)$, the accepted language remains unmodified. Similarly for transitions with label $u\cdot v^i$. Hence, if we keep doing this until there is at most one outgoing transition with label in $(v)^*$ or $u\cdot (v)^*$ for each state $p\in Q$, we will be left with a deterministic automaton. Finally, the subautomaton on the transitions $\delta_v$ gives us the required language.
\end{proof}

\section{Maximal prefix membership}

In \cite{jez_12}, the fully compressed membership problem was shown to be in P for CDFA and in NP for CNFA. A natural extension of this problem is the \emph{fully compressed prefix membership problem}: given as input a CNFA $\mathcal{A}$ and an SLP $\mathbb{X}$, decide if any prefix of $\word{\mathbb{X}}$ is in the language $L(\mathcal{A})$. If $\mathcal{A}$ is not a CDFA, then the problem is clearly still in NP. In order to extend Theorem \ref{jez_compressed} to deterministically solve this problem for CDFA, it turns out that the missing ingredient is Theorem \ref{intersection}.

The following proposition will serve as a base for a binary search approach to the fully compressed prefix membership problem.

\begin{prop}
\label{transition_crossing}
There is a polynomial-time algorithm that, given as input a CDFA $\mathcal{A}$, a transition $\alpha = (p, X, q)\in \delta$ and a non-terminal $W\to U\cdot V$, decides if there is some integer $i\leq \abs{\word{X}}$ such that $U\cdot V[:i]$ determines a path from $q_0$ to $q$, traversing $\alpha$ last. If such an integer exists, then the algorithm also computes $i$.
\end{prop}

\begin{proof}
If such an $i$ exists, then there is a crossing appearance of $X$ in $W\to U\cdot V$. By Lemma \ref{crossing_appearances} we may decide in polynomial time if $W\to U\cdot V$ has a crossing appearance of $X$. If it does, then the algorithm also produces a triple $j, k, l$ such that $\ap(j, k, l)$ encodes all the crossing appearances of $X$ in $W$. Now if such an $i$ exists, then there must be some $m\in \ap(j, k, l)$ such that $W[:m]$ determines a path from $q_0$ to $p$. In particular, we have $W[l:m]\in (W[l:l + k])^*$. So the problem now becomes to find $m$ as then $i = m + \abs{\word{X}}$. Note that since $\mathcal{A}$ is a CDFA, then if $m$ exists, it is unique. Let $\mathcal{A}'$ be a CDFA identical to $\mathcal{A}$ except with $p$ as its only final state. If such an $m$ exists, then we must have $L(\mathcal{A}')\cap W[:l]\cdot (W[l:l+k])^*\neq \emptyset$. By Theorem \ref{intersection}, we may decide in polynomial time if this intersection is non-empty. If it is non-empty then we may also compute integers $n_0, n_1, n$ such that:
\[
L(\mathcal{A}')\cap W[:l]\cdot (W[l:l+k])^* = W[:l]\cdot W[l:l+k]^{n_0}\cdot (W[l:l+k]^n)^*\cdot W[l:l+k]^{n_1}.
\]
Note that $X = W[l:l+k]^r\cdot W[l:l + k']$ for some $r\geq j$ and $k'<k$. If $n\neq 0$, then $W[l:l+k]^n$ determines a path from $p$ to $p$, traversing $\alpha$. Hence we must have $r\leq n$. But if $r\leq n$, then such an $m$ exists if and only if $n_0 + n_1\leq j$. If this condition is satisfied, then:
\[
m = l + (n_0 + n_1)\cdot k
\]
and so:
\[
i = l + (n_0 + n_1)\cdot k + \abs{\word{X}}.
\]
\end{proof}

The following Theorem puts the fully compressed prefix membership problem for CDFA in P. Given a positive answer to an instance of the fully compressed membership problem, it also provides a maximal prefix. Finding the maximal prefix is key to our involutive CNFA to CDFA conversion algorithm in Section \ref{CNFA_to_CDFA}. Furthermore, by combining Theorem \ref{fully_prefix} with the proof of Theorem 7.5 in \cite{me_21}, we may upgrade Theorem \ref{compressed_nei} to a deterministic algorithm in the case that $\mathbb{S} = \emptyset$.

\begin{thm}
\label{fully_prefix}
There is a polynomial-time algorithm that, given as input a CDFA $\mathcal{A}$ and an SLP $\mathbb{W}$, decides if there exists an integer $i$ such that $\mathbb{W}[:i]\in L(\mathcal{A})$. If such an $i$ exists, the algorithm also computes the largest such $i$.
\end{thm}

\begin{proof}
Let $\mathcal{X}$ be the non-terminals of $\mathcal{A}$. If $\abs{\word{\mathbb{W}}}< \min{\{\abs{\word{X}} \mid X\in\mathcal{X}\}}$ then $\mathbb{W}[:i]$ cannot traverse any transition for any $i>0$. If $q_0\in F$, then $i = 0$, otherwise there is no such $i$ and we are done.

Now suppose that $|\word{\mathbb{W}}|\geq |\min{\{\abs{\word{X}} \mid X\in\mathcal{X}\}}|$. The proof is by induction on height. When $||\mathbb{W}|| = 0$ the result is clear. For the inductive hypothesis, suppose the theorem holds for SLPs of height strictly less than $||\mathbb{W}||$.  For each transition $\alpha = (p, X, q)\in \delta$, using Proposition \ref{transition_crossing} we may determine whether there is some integer $j\leq |\word{X}|$ such that $U\cdot V[:j]$ determines a path from $q_0$ to $q$ and traversing $\alpha$ last, where $W\to U\cdot V$ is the root production of $\mathbb{W}$.

If there is no such transition, then, if $i$ exists, we must have $i\leq |\word{U}|$. Since $||U||< ||W||$, by the inductive hypothesis we may decide if there is some $i$ such that $U[:i]\in L(\mathcal{A})$ in polynomial time. Furthermore, if it exists, we may also compute the maximal such $i$. Thus we have decided if there is some $i$ such that $\mathbb{W}[:i]\in L(\mathcal{A})$ and computed the maximal such $i$ in polynomial time.

If there is such a transition $\alpha = (p, X, q)$, then $i>|\word{U}|$ and $U$ determines a path from $q_0$ to $\alpha(-j)$. Now we modify our automaton $\mathcal{A}$ in the following way: remove the transition $\alpha$, add a new state $c$ and two new transitions $(p, X[:-j], c)$ and $(c, X[-j:], q)$. After changing the initial state to $c$, let $\mathcal{A}'$ be the resulting CDFA. We have that $U$ determines a path from $q_0$ to $c$ and so
\[
L(\mathcal{A}) \cap \word{U}\cdot \Sigma^* = \word{U}\cdot L(\mathcal{A}').
\]
Since $||V||<||W||$, by the inductive hypothesis we may determine the maximal index $l$ such that $V[:l]\in L(\mathcal{A}')$ in polynomial time. By construction, $V[:l]\in L(\mathcal{A}')$ if and only if $\mathbb{W}[:|\word{U}| + l]\in L(\mathcal{A})$.
\end{proof}

\begin{cor}
The fully compressed prefix membership problem for CDFA (respectively CNFA) is in P (respectively NP).
\end{cor}

\section{Converting an involutive CNFA to an involutive CDFA}
\label{CNFA_to_CDFA}

Any NFA $\mathcal{A}$ can be converted to a DFA $\mathcal{A}'$ such that $L(\mathcal{A}) = L(\mathcal{A}')$. However, even if $\mathcal{A}'$ is minimal, it may have exponentially many more states and transitions than $\mathcal{A}$. The same holds true for CNFAs. When our input automata are involutive and we want to preserve the image of the language under the map $\pi$, this is no longer true. In this section we present an algorithm to solve the following problem:

\begin{description}
\item[Input] An involutive CNFA $\mathcal{A}$.
\item[Output] An involutive CDFA $\mathcal{A}'$ such that $\pi(L(\mathcal{A})) = \pi(L(\mathcal{A}'))$.
\end{description}

We call this \emph{involutive CNFA to CDFA conversion}. In the classic Stallings algorithm \cite{sta_83}, we are given an involutive NFA as input and we identify pairs of adjacent transitions with the same label until we obtain an involutive DFA which accepts the same language. Each such identification is called a \emph{fold}. Since each fold decreases the number of transitions by one, the number of folds to be performed is bounded above by the number of transitions we started off with. 

Since our input will have compressed transition labels, the number of folds to be performed may be exponential in the input size. In order to overcome this problem, we would like to perform many folds at once. We may do this by what we call \emph{transition folding}.

\subsection{Transition folding}

Let $\mathcal{A}$ be an involutive CNFA. Let $\alpha = (p, X, q)\in \delta$ be a transition. We say a CNFA $\mathcal{A}'$ is obtained from $\mathcal{A}$ by a \emph{transition fold}, or, \emph{folding $\alpha$}, if $\mathcal{A}'$ is obtained in the following way. Let $i$ and $j$ be largest integers such that $X[:i]$ and $X^{-1}[:j]$ determine paths in $\mathcal{A}$ from $p$ and $q$ respectively, not traversing $\alpha$ or $\alpha^{-1}$, and subject to the constraint that $i + j \leq \abs{\word{X}}$. Suppose that $X[:i]$ and $X^{-1}[:j]$ determine paths to $\beta(k)$ and $\gamma(l)$ respectively in $\mathcal{A}$, without traversing $\alpha$ or $\alpha^{-1}$. Let $\beta = (r, Y, s)$ and $\gamma = (t, Z, u)$. We have two cases to consider:

\begin{figure}[t]
\includegraphics[width=8cm]{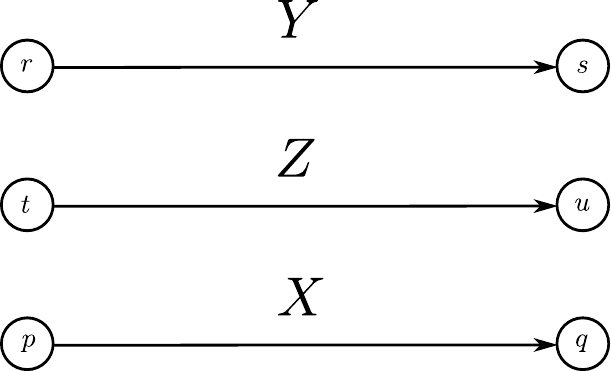}
\caption{Transitions $\beta, \gamma$ and $\alpha$.}
\label{transition_fold_1}
\centering
\end{figure}

\begin{figure}[t]
\includegraphics[width=8cm]{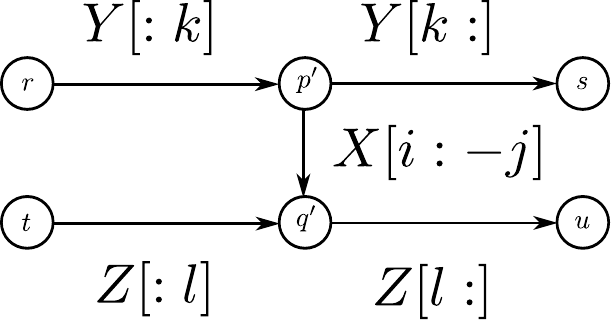}
\caption{Result of folding $\alpha$.}
\label{transition_fold_2}
\centering
\end{figure}

\textbf{Case 1:} $\beta\neq \gamma, \gamma^{-1}$. 

Then we remove the transitions $\alpha, \beta, \gamma$ and their inverses and we add states $p'$ and $q'$ along with the following transitions:
\[
(r, Y[:k], p'), (p', Y[k:], s),
\]
\[
(t, Z[:l], q'), (q', Z[l:], u),
\]
\[
(p', X[i: -j], q'),
\]
and their inverses. See Figures \ref{transition_fold_1} and \ref{transition_fold_2}.

\textbf{Case 2:} $\beta =\gamma$ or $\beta = \gamma^{-1}$. 

Up to replacing transitions by their inverses, we may assume that $\beta = \gamma$ and $k\leq l$. 

\textbf{Subcase 2.1:} $k < l$.

Then we remove the transitions $\alpha, \beta$ and their inverses and add states $p', q'$ and transitions:
\[
(r, Y[:k], p'), (p', Y[k:l], q'), (q', Y[l:], s),
\]
\[
(p', X[i:-j], q'),
\]
and their inverses.

\textbf{Subcase 2.2:} $k = l$.

Let $m$ be the largest integer such that $\word{X[i:i+m]} = \word{X^{-1}[j:j+m]}$. Then we remove the transitions $\alpha, \beta$ and their inverses and add states $p', q'$ and transitions:
\[
(r, Y[:k], p'), (p', Y[k:], s),
\]
\[
(p', X[i:i+m], q'), (q', X[i+m:-j-m], q'),
\]
and their inverses.

Finally, we remove any transitions with empty label and identify the corresponding states. If $X[:i]$ instead determines a path to a state $q'$, then we may add a transition $\beta = (q', \epsilon, q')$ so that $X[:i]$ determines a path to $\beta(0)$ and folding $\alpha$ is defined as above. Similarly for $X^{-1}[:j]$.

We will write $\mathcal{A}\rightarrow \mathcal{A}'$ to denote that $\mathcal{A}'$ is obtained from $\mathcal{A}$ by a transition fold and $\mathcal{A}\xrightarrow{\alpha} \mathcal{A}'$ to denote that $\mathcal{A}'$ is obtained from $\mathcal{A}$ be folding the transition $\alpha$. We will call a transition fold $\mathcal{A}\rightarrow \mathcal{A}'$ \emph{complete} if $i+j = \abs{\word{X}}$, \emph{incomplete} otherwise. Note that if $\mathcal{A}$ without the transitions $\alpha$ and $\alpha^{-1}$ is not deterministic, then there may be more than one CNFA that is obtained from $\mathcal{A}$ by folding the transition $\alpha$.

The following two lemmas are by definition of a transition fold.

\begin{lma}
\label{fold_properties}
If $\mathcal{A}\rightarrow \mathcal{A}'$, then:
\begin{enumerate}
\item $\pi(L(\mathcal{A}')) = \pi(L(\mathcal{A}))$,
\item $\abs{\delta'} \leq \abs{\delta} + 4$ and $\abs{\mathcal{X}'} = \abs{\mathcal{X}}$,
\item if each transition label of $\mathcal{A}$ is freely reduced, then so is each transition label of $\mathcal{A}'$.
\end{enumerate}
\end{lma}

\begin{lma}
\label{incomplete_fold}
Suppose that the subautomaton of $\mathcal{A}$ without the transitions $\alpha$ and $\alpha^{-1}$ is a CDFA. Then if $\mathcal{A}\xrightarrow{\alpha}\mathcal{A}'$ is an incomplete transition fold, then $\mathcal{A}'$ is a CDFA.
\end{lma}

The following theorem follows directly from Theorem \ref{fully_prefix} and the definition of a transition fold.

\begin{thm}
\label{transition_fold}
There is a polynomial-time algorithm that, given as input an involutive CDFA $\mathcal{A}$ and a transition $\alpha\in \delta$ such that the subautomaton of $\mathcal{A}$ without the transitions $\alpha$ and $\alpha^{-1}$ is a CDFA, computes an involutive CNFA $\mathcal{A}'$ such that $\mathcal{A}\xrightarrow{\alpha} \mathcal{A}'$.
\end{thm}

\subsection{Minimalistic involutive CNFAs}

Throughout our algorithm, we will require that our CNFAs be of a particular form. An involutive CNFA $\mathcal{A}$ is \emph{minimalistic} if the following holds:
\begin{enumerate}
\item There are no transitions $(p, X, q)\in \delta$ such that $\word{X}$ is not freely reduced.
\item There are no states $p\in Q$ such that $p\neq q_0$, $p\notin F$ and $p$ has at most two incoming and at most two outgoing transitions.
\end{enumerate}

Every CNFA can be transformed into an equivalent minimalistic one. The following lemma says that we can do this efficiently and without adding any complexity.

\begin{lma}
\label{minimalistic_algorithm}
There is a polynomial-time algorithm that, given as input an involutive CNFA $\mathcal{A}$, computes an involutive CNFA $\mathcal{A}'$ such that:
\begin{enumerate}
\item $\mathcal{A}'$ is minimalistic,
\item $\pi(L(\mathcal{A}')) = \pi(L(\mathcal{A}))$,
\item $\abs{\delta'} + \abs{\mathcal{X}'}\leq \abs{\delta} + \abs{\mathcal{X}}$.
\end{enumerate}
\end{lma}

\begin{proof}
Let $\mathcal{A}$ be the input involutive CNFA. For each state $q\in Q\setminus\{q_0, F\}$ such that $q$ has only one incoming and one outgoing transition, remove $q$ and the two transitions. Then, for each state $q\in Q\setminus\{q_0, F\}$ such that $q$ has two incoming transitions $(p_1, X, q)$, $(p_2, Y, q)$ and two outgoing transitions $(q, X^{-1}, p_1)$, $(q, Y^{-1}, p_2)$, remove these transitions and the state $q$ and add two non-terminals $Z\to X\cdot Y^{-1}$ and $Z^{-1}\to Y\cdot X^{-1}$ and two new transitions $(p_1, Z, p_2)$ and $(p_2, Z^{-1}, p_1)$. The language is clearly preserved under these operations and the output $\mathcal{A'}$ will be minimalistic by construction.
\end{proof}

The following lemma will be used for our inductive hypothesis when bounding the number of transition folds to be performed in Algorithm \ref{CDFA_algorithm}.

\begin{lma}
\label{transition_bound}
Let $\mathcal{A}$ be an minimalistic involutive CNFA. Then:
\begin{align*}
\frac{\abs{\delta}}{2}&\leq 3\cdot \max{\left\{0, \frac{\abs{\delta}}{2} - \abs{Q}\right\}} + \abs{F} + \abs{\partial \mathcal{A}} + 1,\\
\abs{Q} &\leq 2\cdot \max{\left\{0, \frac{\abs{\delta}}{2} - \abs{Q}\right\}} + \abs{F} + \abs{\partial \mathcal{A}} + 1.
\end{align*}
\end{lma}

\begin{proof}
Denote by $M = (\{q_0\}\cup F) - \partial \mathcal{A}$. If $\deg(q)$ denotes the number of incoming and outgoing transitions from the state $q$, then we have $\deg(q)\geq 6$ for all $q\in Q - (M\sqcup \partial\mathcal{A})$, provided $\delta\neq \emptyset$. Thus:
\begin{align*}
\frac{\abs{\delta}}{2} &= \sum_{q\in Q}\frac{\deg(q)}{4} \\
	&\geq \frac{3}{2}\cdot (\abs{Q} - \abs{M} - \abs{\partial\mathcal{A}}) + \abs{M} + \frac{1}{2}\cdot \abs{\partial\mathcal{A}} \\
	&= \frac{3}{2}\cdot \abs{Q} - \frac{1}{2}\cdot \abs{M} - \abs{\partial\mathcal{A}}.
\end{align*}
Since $\abs{M} + \abs{\partial\mathcal{A}}\leq \abs{F} + 1$, by subtracting $\frac{3}{4}\cdot \abs{\delta}$ from both sides and multiplying by $-2$, we obtain the first inequality. Similarly, by subtracting $\abs{Q}$ from both sides and multiplying by $2$ we obtain the second inequality.
\end{proof}

\subsection{The algorithm}

We are now ready to present our involutive CNFA to involutive CDFA conversion algorithm. See Algorithm \ref{CDFA_algorithm} for the outline. Denote by
\[
k(\mathcal{A}) = 3\cdot \max{\left\{0, \frac{\abs{\delta}}{2} - \abs{Q}\right\}} + \abs{F} + \abs{\partial \mathcal{A}} + 1
\]
and 
\[
n(\mathcal{A}) = \sum_{(p, X, q)\in \delta}\frac{\abs{\word{X}}}{2}.
\]
The following proposition shows that the algorithm solves the involutive CNFA to CDFA conversion problem and also bounds the number of transition folds performed in terms of $k(\mathcal{A})$ and $n(\mathcal{A})$. The main idea is that, by our choice of the order in which the transition folds are performed, each complete transition fold gets rid of a definite proportion of the CNFA. Note that the size of the input to Algorithm \ref{CDFA_algorithm} is bounded below by $\log_2(n(\mathcal{A}))$.

\begin{algorithm}
\caption{Involutive CNFA to CDFA conversion}\label{CDFA_algorithm}
\begin{enumerate}
\item\label{itm:initiate} Make $\mathcal{A}$ minimalistic. Let $i = 0$ and $\mathcal{A} = \mathcal{A}_0$.
\item\label{itm:subfold} Let $\alpha_i = (p, X, q)\in \delta_i$ such that $\abs{\word{X}}$ is maximal over all transitions. If $\mathcal{A}_i' = (Q, \Sigma\sqcup\Sigma^{-1}, \mathcal{X}, \delta - \{\alpha_i, \alpha_i^{-1}\}, \mathcal{P}, q_0, F\cup\{p, q\})$ is a CDFA, then continue to \ref{itm:fold}. If not, then convert $\mathcal{A}_i'$ to a CDFA and let $\mathcal{A}_i''$ be the output. Now let $\mathcal{A}_i'''$ be obtained from $\mathcal{A}_i''$ by adding $\alpha_i$ and $\alpha_i^{-1}$ to its transition set and removing $p$ and $q$ from its final states if they were not final states in $\mathcal{A}_i$.
\item\label{itm:fold} Fold $\alpha_i$ and make the resulting CNFA minimalistic. Let $\mathcal{A}_{i+1}$ be the output. If $\mathcal{A}_{i+1}$ is a CDFA then terminate with output $\mathcal{A}_{i+1}$. If not, then let $i := i+1$ and go back to \ref{itm:subfold}.
\end{enumerate}
\end{algorithm}

\begin{prop}
\label{fold_bound}
Let $\mathcal{A}$ be a non-trivial involutive CNFA and let $k = k(\mathcal{A})$ and $n = n(\mathcal{A})$. Then Algorithm \ref{CDFA_algorithm} performs at most $(2\cdot k\cdot \log(n))^k$ many transition folds and outputs a CDFA $\mathcal{A}_m$ such that $\pi(L(\mathcal{A}_m)) = \pi(L(\mathcal{A}))$.
\end{prop}

\begin{proof}
We first note that $\frac{\abs{\delta}}{2}\leq k$ by Lemma \ref{transition_bound} and that $k(\mathcal{A}')\geq 3$ for all CNFAs $\mathcal{A}'$. We also note that for all $i\geq 2$:
\[
\frac{1}{\log\left(\frac{i}{i-1}\right)}\leq i
\]
so that it suffices to show that Algorithm \ref{CDFA_algorithm} performs at most $\left(2\cdot \log_{\frac{k}{k-1}}(n)\right)^k$ many transition folds.

The proof is by induction on $k$. The base case when $k = 3$ is trivial as $\mathcal{A}$ has only two states and one transition. So now suppose that the result is true for all minimalistic involutive CNFAs $\mathcal{A}'$ with $k(\mathcal{A}')<k(\mathcal{A})$.

We have $k(\mathcal{A}_i)\leq k(\mathcal{A})$ for all $i$ by the definition of a transition fold. By Lemma \ref{incomplete_fold}, $\mathcal{A}_i$ is obtained from $\mathcal{A}_{i-1}'''$ by a complete transition fold for all $i<m$. Thus, since $\abs{\word{X}}\geq \frac{1}{k}\cdot n(\mathcal{A}_i)$, it follows that
\[
n(\mathcal{A}_i)\leq \frac{k-1}{k}\cdot n(\mathcal{A}_{i-1})\leq \left(\frac{k-1}{k}\right)^i\cdot n(\mathcal{A})
\]
for all $i<m$. But then we get that
\[
m \leq \log_{\frac{k}{k-1}}(n(\mathcal{A})).
\]
Suppose that $p\notin F$, then $p\notin \partial \mathcal{A}_i'$ as $\mathcal{A}_i$ is minimalistic. Similarly for $q$. Thus, $k(\mathcal{A}_i')\leq k(\mathcal{A}) - 1$ and the inductive hypothesis applies to $\mathcal{A}_i'$. So $\mathcal{A}_i''$ is obtained from $\mathcal{A}_i'$ by performing at most
\[
\left(2\cdot \log_{\frac{k-1}{k-2}}\left(\left(\frac{k-1}{k}\right)^i\cdot n(\mathcal{A})\right)\right)^{k-1}\leq 2^{k-1}\cdot \left(\log_{\frac{k}{k-1}}\left(n(\mathcal{A})\right) - i\right)^{k-1}
\]
many transition folds. Finally, summing over everything we get a bound of:
\[
m + \sum_{i=0}^{m-1} 2^{k-1}\cdot\left(\log_{\frac{k}{k-1}}\left(n(\mathcal{A})\right) - i\right)^{k-1}\leq \left(2\cdot \log_{\frac{k}{k-1}}(n(\mathcal{A}))\right)^k
\]
many transition folds.

The fact that $\mathcal{A}_m$ is a CDFA follows from the algorithm. The fact that $\pi(L(\mathcal{A}_m)) = \pi(L(\mathcal{A}))$ follows from Lemmas \ref{fold_properties} and \ref{minimalistic_algorithm}. Thus, the algorithm is correct.
\end{proof}

Finally, putting everything together, we may provide complexity bounds for the compressed involutive CNFA to CDFA conversion problem.

\begin{thm}
\label{conversion_complexity}
There is an algorithm that, given as input an involutive CNFA $\mathcal{A}$, computes an involutive CDFA $\mathcal{A}'$ such that $\pi(L(\mathcal{A}')) = \pi(L(\mathcal{A}))$ in $O\left(n^{O(\abs{\delta})}\right)$ time.
\end{thm}

\begin{proof}
By Lemma \ref{minimalistic_algorithm}, Step \ref{itm:initiate} requires polynomial time. By Lemma \ref{transition_bound} and Proposition \ref{fold_bound}, Algorithm \ref{CDFA_algorithm} performs Step \ref{itm:fold} at most $O\left(n^{O(\abs{\delta})}\right)$ many times. By Lemma \ref{incomplete_fold}, Theorem \ref{transition_fold} and Lemma \ref{minimalistic_algorithm}, Step \ref{itm:fold} requires time polynomial in $n$. So overall, Algorithm \ref{CDFA_algorithm} is in $O\left(n^{O(\abs{\delta})}\right)$.
\end{proof}

Applying the main results from each section to the context of free groups, we prove our main theorem.

\begin{thm}
\label{main_subgroup}
Given as input a collection $\mathbb{W}, \mathbb{W}_1, \ldots, \mathbb{W}_k$ of SLPs over $\Sigma\sqcup\Sigma^{-1}$, we may do the following in $O\left(n^{O(k)}\right)$ time:
\begin{enumerate}
\item Compute a compressed Stallings automaton for $H = \langle\pi(\word{\mathbb{W}_1}), \ldots, \pi(\word{\mathbb{W}_k})\rangle<F(\Sigma)$.
\item Compute an SLP $\mathbb{X}$ such that $\pi(\word{\mathbb{W}})\in H\cdot \pi(\word{\mathbb{X}})$ and $\abs{\word{\mathbb{X}}}$ is smallest possible.
\item Compute an SLP $\mathbb{X}$ such that $\langle\pi(\word{\mathbb{W}})\rangle\cap H = \langle \pi(\word{\mathbb{X}})\rangle$.
\end{enumerate}
\end{thm}

\begin{proof}
Let $W$ be the root non-terminal for $\mathbb{W}$ and $W_i$ the root non-terminals for $\mathbb{W}_i$ for each $i$. Let $\mathcal{X}$ be the union of all the non-terminals appearing in each $\mathbb{W}_i$ and $\mathcal{P}$ the union of all the productions. Consider the CNFA
\[
\mathcal{A} = (\{q_0\}, \Sigma\sqcup\Sigma^{-1}, \mathcal{X}, \{(q_0, W_i, q_0), (q_0, W_i^{-1}, q_0)\}_{i=1}^k, \mathcal{P}, q_0, \{q_0\}).
\]
We have $\pi(L(\mathcal{A})) = H$. By Theorem \ref{conversion_complexity} we may compute an minimalistic involutive CDFA $\mathcal{A}'$ such that $\pi(L(\mathcal{A}')) = \pi(L(\mathcal{A}))$ in $O(n^{O(k)})$ time. At no point in the algorithm can an unreachable state be created. Combined with the fact that $\mathcal{A}'$ is minimalistic, it follows that $\mathcal{A}'$ is a compressed Stallings automaton for the subgroup $H$.

By Theorem 3.3 in \cite{saul_08}, we may compute an SLP $\mathbb{W}'$ such that $\word{\mathbb{W}'}$ is freely reduced and $\pi(\word{\mathbb{W}'}) = \pi(\word{\mathbb{W}})$. Let $i$ be the maximal integer such that $\mathbb{W}'[:i]$ determines a path from the start state of $\mathcal{A}'$ to any other state. Let $p$ be this state. By Theorem \ref{fully_prefix}, we may compute $p$ and $i$ in polynomial time. We may also compute the shortest path from the start state to $p$ in polynomial time. Let $X_1\cdot \ldots\cdot X_m$ be the sequence of labels determined by this path. Then if $\mathbb{X}$ is an SLP with $\word{\mathbb{X}} = \word{X_1}\cdot \ldots\cdot \word{X_m}\cdot \word{\mathbb{W}'}[i:]$, then $\pi(\word{\mathbb{W}})\in H\cdot \pi(\word{X_1\cdot \ldots\cdot X_m\cdot \mathbb{W}'[i:]})$ and $\abs{\word{\mathbb{X}}}$ is minimal possible.

By Corollary 3.6 in \cite{saul_08}, we may compute in polynomial time SLPs $\mathbb{U}$ and $\mathbb{V}$ such that $\word{\mathbb{U}}\cdot \word{\mathbb{V}}\cdot \word{\mathbb{U}^{-1}} = \word{\mathbb{W}'}$ with $\abs{\word{\mathbb{U}}}$ maximal possible. By Theorem \ref{fully_prefix}, we may decide if there is a transition $\alpha\in \delta'$ and an integer $i$ such that $\word{\mathbb{U}}$ determines a path from the start state of $\mathcal{A}'$ to $\alpha(i)$. Furthermore, if it does, we may compute $\alpha = (p, X, q)$ and $i$ in polynomial time. Then we modify $\mathcal{A}'$ as follows: add a new state $p'$ and new transitions $(p, X[:i], p')$, $(p', X^{-1}[-i:], p)$, $(p', X[i:], q)$ and $(q, X^{-1}[:i], p')$. Remove the transitions $\alpha$ and $\alpha^{-1}$ and make $p'$ the initial and final state. Let $\mathcal{A}''$ be the resulting CDFA. Now $\langle\pi(\word{\mathbb{W}})\rangle\cap H = \langle \pi(\word{\mathbb{W}})^m\rangle$ if and only if $(\word{\mathbb{V}})^*\cap L(A'') = (\word{\mathbb{V}}^m)^*$. Finally, we may compute $m$ in polynomial time by Theorem \ref{intersection}.
\end{proof}

\begin{cor}
\label{main_cor}
The fully compressed membership problem for $k$-generated subgroups of a free group  is in P.
\end{cor}

\bibliographystyle{abbrv}
\bibliography{bibliography}

\begin{thebibliography}{10}

\bibitem{avenhaus_89}
J.~Avenhaus and D.~Wi:Gbmann.
\newblock Using rewriting techniques to solve the generalized word problem in
  polycyclic groups.
\newblock In {\em Proceedings of the ACM-SIGSAM 1989 International Symposium on
  Symbolic and Algebraic Computation}, ISSAC '89, page 322–337, New York, NY,
  USA, 1989. Association for Computing Machinery.

\bibitem{benois_69}
M.~Benois.
\newblock Parties rationnelles du groupe libre.
\newblock {\em Comptes rendus de l'Académie des Sciences}, 269:1188--1190,
  1969.

\bibitem{gasieniec_96}
L.~Gasieniec, M.~Karpinski, W.~Plandowski, and W.~Rytter.
\newblock Efficient algorithms for lempel-ziv encoding.
\newblock In R.~Karlsson and A.~Lingas, editors, {\em Algorithm Theory ---
  SWAT'96}, pages 392--403, Berlin, Heidelberg, 1996. Springer Berlin
  Heidelberg.

\bibitem{grunschlag_99}
Z.~Grunschlag.
\newblock {\em Algorithms in geometric group theory}.
\newblock PhD thesis, University of California at Berkely, 1999.

\bibitem{gurevich_05}
Y.~Gurevich and P.~Schupp.
\newblock Membership problem for the modular group.
\newblock Technical Report MSR-TR-2005-92, July 2005.

\bibitem{hagenah_00}
C.~Hagenah.
\newblock {\em Gleichungen mit regul{\"{a}}ren Randbedingungen {\"{u}}ber
  freien Gruppen}.
\newblock PhD thesis, University of Stuttgart, Germany, 2000.

\bibitem{jez_12}
A.~J{\.{e}}z.
\newblock "{Compressed membership for NFA (DFA) with compressed labels is in NP
  (P)}".
\newblock In C.~D{\"u}rr and T.~Wilke, editors, {\em 29th International
  Symposium on Theoretical Aspects of Computer Science (STACS 2012)}, volume~14
  of {\em Leibniz International Proceedings in Informatics (LIPIcs)}, pages
  136--147, Dagstuhl, Germany, 2012. Schloss Dagstuhl--Leibniz-Zentrum fuer
  Informatik.

\bibitem{kapovich_02}
I.~Kapovich and A.~Myasnikov.
\newblock Stallings foldings and subgroups of free groups.
\newblock {\em Journal of Algebra}, 248(2):608--668, 2002.

\bibitem{lifshits_07}
Y.~Lifshits.
\newblock Processing compressed texts: A tractability border.
\newblock In B.~Ma and K.~Zhang, editors, {\em Combinatorial Pattern Matching},
  pages 228--240, Berlin, Heidelberg, 2007. Springer Berlin Heidelberg.

\bibitem{me_21}
M.~Linton.
\newblock On the intersections of finitely generated subgroups of free groups:
  reduced rank to full rank, 2021.

\bibitem{lohrey_12}
M.~Lohrey.
\newblock {Algorithmics on SLP-compressed strings: a survey}.
\newblock {\em Groups - Complexity - Cryptology}, 4(2):241--299, 2012.

\bibitem{lohrey_21_compression}
M.~Lohrey.
\newblock Compression techniques in group theory.
\newblock In L.~De~Mol, A.~Weiermann, F.~Manea, and D.~Fern{\'a}ndez-Duque,
  editors, {\em Connecting with Computability}, pages 330--341, Cham, 2021.
  Springer International Publishing.

\bibitem{lohrey_21}
M.~Lohrey.
\newblock {Subgroup Membership in GL(2,Z)}.
\newblock In M.~Bl\"{a}ser and B.~Monmege, editors, {\em 38th International
  Symposium on Theoretical Aspects of Computer Science (STACS 2021)}, volume
  187 of {\em Leibniz International Proceedings in Informatics (LIPIcs)}, pages
  51:1--51:17, Dagstuhl, Germany, 2021. Schloss Dagstuhl -- Leibniz-Zentrum
  f{\"u}r Informatik.

\bibitem{lohrey_08}
M.~Lohrey and B.~Steinberg.
\newblock The submonoid and rational subset membership problems for graph
  groups.
\newblock {\em Journal of Algebra}, 320(2):728--755, 2008.
\newblock Computational Algebra.

\bibitem{myasnikov_11}
A.~Myasnikov, A.~Ushakov, and D.~W. Won.
\newblock The word problem in the baumslag group with a non-elementary dehn
  function is polynomial time decidable.
\newblock {\em Journal of Algebra}, 345(1):324--342, 2011.

\bibitem{plandowski_99}
W.~Plandowski and W.~Rytter.
\newblock {\em Complexity of Language Recognition Problems for Compressed
  Words}, pages 262--272.
\newblock Springer Berlin Heidelberg, Berlin, Heidelberg, 1999.

\bibitem{romankov_99}
V.~Roman’kov.
\newblock On the occurence problem for rational subsets of a group.
\newblock In {\em International Conference on Combinatorial and Computational
  Methods in Mathematics}, pages 76--81, 1999.

\bibitem{saul_08}
S.~Schleimer.
\newblock Polynomial-time word problems.
\newblock {\em Commentarii Mathematici Helvetici}, 83:741--765, 2006.

\bibitem{sta_83}
J.~Stallings.
\newblock Topology of finite graphs.
\newblock {\em Inventiones mathematicae}, 71:551--565, 1983.

\end{thebibliography}

\end{document}